\documentclass[12pt, a4paper]{amsart}
\usepackage{amsmath}
\usepackage{geometry,amsthm,graphics,tabularx,amssymb,shapepar}
\usepackage[cmyk]{xcolor}
\usepackage{appendix}
\usepackage{amscd}
\usepackage[all]{xypic}
\usepackage{ulem}
\usepackage{bm}
\usepackage[colorlinks,citecolor=black,linkcolor=black,anchorcolor=black]{hyperref}

\makeatletter
\newcommand*{\rom}[1]{\expandafter\@slowromancap\romannumeral #1@}
\makeatother

\newcommand{\BN}{{\mathbb {N}}}

\newcommand{\RC}{{\mathrm {C}}}

\newcommand{\RH}{{\mathrm {H}}}

\newcommand{\RT}{{\mathrm {T}}}
\newcommand{\RU}{{\mathrm {U}}}

\newcommand{\End}{{\mathrm{End}}}

\newcommand{\Hom}{{\mathrm{Hom}}}

\renewcommand{\Im}{{\mathrm{Im}}}

\newcommand{\Ker}{{\mathrm{Ker}}}

\newcommand{\Lie}{{\mathrm{Lie}}}

\newcommand{\SL}{{\mathrm{SL}}}

\newcommand{\wh}{\widehat}

\newcommand{\ov}{\overline}

\newcommand{\con}{\textit{C}}

\newcommand{\dif}{\operatorname{d}}

\newcommand{\Tr}{\operatorname{Tr}}

\newcommand{\sgn}{\operatorname{sgn}}

\newcommand{\g}{\mathfrak g}

\renewcommand{\k}{\mathfrak k}
\newcommand{\h}{\mathfrak h}

\newcommand{\q}{\mathfrak q}

\renewcommand{\b}{\mathfrak b}

\renewcommand{\u}{\mathfrak u}

\renewcommand{\l}{\mathfrak l}

\newcommand{\s}{\mathfrak s}

\renewcommand{\o}{\mathfrak o}

\newcommand{\z}{\mathfrak z}

\newcommand{\C}{\mathbb{C}}

\newcommand{\abs}[1]{\lvert#1\rvert}

\newcommand{\be}{\begin {equation}}
\newcommand{\ee}{\end {equation}}
\newcommand{\bee}{\begin {equation*}}
\newcommand{\eee}{\end {equation*}}

\renewcommand{\mid}{\,:\,}

\theoremstyle{plain}

\theoremstyle{plain}

\theoremstyle{plain}
\newtheorem{lem}{Lemma}[section]

\newtheorem{thml}[lem]{Theorem}
\newtheorem{leml}[lem]{Lemma}

\newtheorem{rmkl}[lem]{Remark}

\theoremstyle{plain}

\theoremstyle{plain}

\theoremstyle{remark}

\theoremstyle{remark}

\theoremstyle{definition}

\numberwithin{equation}{section}

\title{Hausdorffness of certain nilpotent cohomology spaces}

\author{Fabian Januszewski}
\address{Institut f\"{u}r Mathematik, Fakult\"{a}t EIM, Paderborn University, Warburger Str.\ 100, 33098 Paderborn, Germany}
\email{fabian.januszewski@math.uni-paderborn.de}

\author{Binyong Sun}
\address{Institute for Advanced Study in Mathematics and New Cornerstone Science Laboratory, Zhejiang University,  Hangzhou, 310058, China}
\email{sunbinyong@zju.edu.cn}

\author{Hao Ying}
\address{School of Mathematical Sciences, Zhejiang University,  Hangzhou, 310058, China}
\email{yhmath@zju.edu.cn}

\begin{document}
		
	\maketitle
		
	\begin{abstract}
		Let $(\pi,V)$ be a smooth representation of a compact Lie group $G$ on a quasi-complete locally convex complex  topological vector space.  We show that the Lie algebra cohomology space $\RH^\bullet(\u, V)$ and the Lie algebra homology space $\RH_\bullet(\u, V)$ are both Hausdorff, where $\u$ is the nilpotent radical of a parabolic subalgebra of the complexified Lie algebra $\g$ of $G$.
	\end{abstract}

\section{Introduction}
Let $G$ be a Lie group. 
In our context, a representation of $G$ is a quasi-complete, locally convex, Hausdorff complex topological vector space $V$, together with a continuous linear action
\be\label{action}
G\times V\rightarrow V. 
\ee
We say that a representation $V$ 
of $G$ is smooth if the action map \eqref{action} is smooth as a map between (possibly infinite-dimensional) smooth manifolds (\cite{GN07}). To emphasize the action, we will denote a representation by the pair $(\pi, V)$, where $\pi$ refers to the action.

Let $\g$ be a finite-dimensional complex Lie algebra. A continuous $\g$-module is a quasi-complete, locally convex, Hausdorff complex topological vector space $V$, together with a continuous Lie algebra action
\be
\g\times V\rightarrow V.
\ee

For a continuous $\g$-module $V$, recall that the $\g$-cohomology is computed by the total complex
\[
\Hom(\wedge^\bullet\g,V)=\wedge^\bullet (\g ^*) \otimes V,
\]
where $\,^*$ indicates the dual space and $\wedge^\bullet$ indicates the exterior algebra. The complex $\Hom(\wedge^\bullet \g,V)$ carries a natural topology and the coboundary map $\dif\,$ is continuous for this topology. The cohomology space $\RH^\bullet(\g, V)$ inherits the natural subquotient topology.

Dually, the total complex for $\g$-homology is
\[
\wedge^\bullet\g\otimes V
\]
with its natural topology and the continuous boundary map $\partial$. The homology space $\RH_\bullet(\g,V)$ again inherits the natural subquotient topology.
Both $\RH^\bullet(\g,V)$ and $\RH_\bullet(\g,V)$ are not necessarily Hausdorff, since the images of $\dif\,$ and $\partial$ are not necessarily closed.

From this point onwards, we assume that $G$ is a real reductive group and $\g$ is its complexified Lie algebra. Fix a Cartan involution $\theta$ of $G$. Denote by $K:=G^\theta$ the fixed point subgroup of $\theta$, which is a maximal compact subgroup of $G$. Let $\q=\l\oplus\u$ be a $\theta$-stable parabolic subalgebra of $\g$ with nilpotent radical $\u$ and Levi factor $\l:=\q\cap\bar\q$. 
Here and henceforth, $``\,\ov{\phantom a}\,"$ over a Lie subalgebra of $\g$ denotes its complex conjugation with respect to the real form $\Lie(G)\subseteq \g$, where $\Lie(G)$ is the Lie algebra of $G$. Let $L:=N_G(\q)=N_G(\ov\q)$ (the normalizers) be the Levi subgroup. Note that $\l$ is the complexified Lie algebra of $L$.

Given a $(\g,K)$-module $M$ of $G$, a globalization of $M$ is defined to be a representation $V$ of $G$ together with a $(\g,K)$-module isomorphism between $M$ and the underlying $(\g,K)$-module of $V$.

When $M$ has finite length (in this case $M$ is called a Harish-Chandra module), it has  four  canonical globalizations: the minimal globalization $M^{\min}$, the Casselman-Wallach globalization $M^\infty$, the distribution globalization $M^{-\infty}$, and the maximal globalization $M^{\max}$. These globalizations are smooth representations of $G$ on Fr\'echet or dual  Fr\'echet spaces, and they fit into a sequence of inclusions
\[
M\subset M^{\min}\subset M^\infty\subset M^{-\infty}\subset M^{\max}.
\]
For $\alpha\in\{\min,\infty,-\infty,\max\}$, Vogan conjectured that the $\u$-cohomology $\RH^\bullet(\u,M^\alpha)$ is Hausdorff, which implies that $\RH^\bullet(\u,M^\alpha)$ is a smooth representation of $L$. For more details about Vogan's conjecture, we refer to \cite[Conjecture 10.3]{Vo}.

Obviously, Vogan's conjecture is trivially true if $G$ is compact, since all Harish-Chandra modules are finite-dimensional in this case. Bratten and Corti proved in \cite{Bratten1998,Br06} that Vogan's conjecture holds when $M^\alpha$ is the minimal globalization or the maximal globalization. For related research on Hausdorffness of Lie algebra cohomologies, see \cite{Bratten1998,HT98,LLY21,Wong92,Wong99}.

Given an arbitrary smooth representation $V$ of a real reductive group $G$,
which is naturally a continuous $\g$-module,
we may ask more generally whether the cohomology space $\RH^\bullet(\u, V)$ is Hausdorff. When $G$ is compact, we establish a stronger result that ensures the Hausdorffness of $\RH^\bullet(\u, V)$, as stated in the following theorem.

\begin{thml}\label{main}
	Suppose that $G$ is compact. Let $V$ be a smooth representation of $G$. 

	\noindent 
	(a) Denote by
	\[
	\dif:  \wedge^\bullet (\u ^*)\otimes V\rightarrow \wedge^\bullet (\u ^* )\otimes V
	\]
	and
	\[
	\partial:\wedge^\bullet\u \otimes V\rightarrow \wedge^\bullet\u \otimes V
	\]
	the coboundary and boundary maps respectively, both of which are $L$-equivariant with respect to the natural actions of $L$. Then both of the inclusion maps
	\begin{equation}
		\Im \dif \hookrightarrow\Ker \dif
		\label{eq:inclusion1}
	\end{equation}
	and
	\begin{equation}
		\Im\,\partial \hookrightarrow\Ker\, \partial
		\label{eq:inclusion2}
	\end{equation}
	admit a degree-preserving $L$-equivariant  continuous linear splitting. 
	
	\noindent
	(b) The Lie algebra cohomology space $\RH^\bullet(\u,V)$ and the Lie algebra homology space $\RH_\bullet(\u,V)$ are Hausdorff and quasi-complete. Moreover, with the natural actions of $L$, $\RH_\bullet(\u,V)$ and $\RH^\bullet(\ov\u,V)$ are   smooth representations of $L$ that are isomorphic to each other.
\end{thml}

In the above theorem, the images $\Im \dif$ and $\Im \,\partial $, as well as the kernels  $\Ker \dif$ and $\Ker\,\partial$,  are all equipped with the subspace topologies.

In Section \ref{sec:proof}, we will utilize Dirac cubic operators to construct explicit splittings for the inclusion maps \eqref{eq:inclusion1} and \eqref{eq:inclusion2} in Theorem \ref{main}.

Our results are related to Vogan's conjecture through Hochschild-Serre spectral sequence, which is one motivation of our work. 
Recall that $V$ is a smooth representation of a real reductive group $G$. Note that $\u\cap \k $ is a nilpotent subalgebra of $\k$, where $\k=\g^\theta$ is the complexified Lie algebra of $K$.
  
Ignoring the topology, there exists a convergent spectral sequence $\{E_r^{p,q}\}_{r\geq 0}\Rightarrow \RH^{p+q}(\u,V)$ in the category of $\l\cap\k$-modules, which is called the Hochschild-Serre spectral sequence. The $E_1$-terms are given by $E_1^{p,q}=\RH^{n(p,q)}(\u\cap\k, V)\otimes X_p$. Here $X_p$ is a finite-dimensional vector space and $n(p,q)$ is a certain integer depending on $p$ and $q$. For more details on the Hochschild-Serre spectral sequence, we refer to \cite[Chapter V, Section 10]{KV}. 

Equipped with natural subquotient topology on $E_r$, the Hochschild-Serre spectral sequence suggests that one may prove Vogan's conjecture by establishing the Hausdorffness of $E_r$-terms for every $r\geq 0$. Along this strategy, we expect that the total cohomology space $\RH^\bullet(\u,V) $ is Hausdorff for certain smooth representations $V$ whose underlying $(\g,K)$-module are not necessarily of finite length. The Hausdorffness of $E_0$-terms is clear from the construction, and our results show that the $E_1$-terms are indeed Hausdorff. However, for $r\geq2$, the Hausdorffness of $E_r$-terms is still unresolved. We hope to investigate this in future research.

\begin{rmkl}
	For some non-compact real reductive groups $G$ and some  smooth representations $V$ of $G$, the cohomology space $\RH^\bullet(\ov\u,V)$ may not be isomorphic to the homology space  $\RH_\bullet(\u,V)$. For instance, a counterexample is provided in \cite[Section 8]{HPR05}: Let $ G = \SL_2(\mathbb{R}) $ and $ \mathfrak{q} = \mathfrak{so}_2 \oplus \mathfrak{u} $, where $ \mathfrak{u} = \mathbb{C}
	\begin{bmatrix}
		1 & \sqrt{-1} \\
		\sqrt{-1} & -1
	\end{bmatrix}$.
	Let $ M $ be a Harish-Chandra module for $G$ that is a nontrivial extension of a highest weight module of highest weight $-2$  by a trivial module. Then the dimension of $\RH^\bullet(\ov\u,M)$ is $3$ while the dimension of $\RH_\bullet(\u,M)$ is $1$.  Thus the  comparison theorem (cf.\, \cite[Theorem 1]{Bratten1998}) implies that $\RH^\bullet(\ov\u,M^{\min})\ncong \RH_\bullet(\u,M^{\min})$.
\end{rmkl}

\textbf{Acknowledgment:} The authors would like to thank Wei Xiao for suggesting the references \cite{Huang06,HPR05,Ko99} on the cubic Dirac operator. F.\,J. acknowledges support by the Deutsche Forschungsgemeinschaft (DFG, German Research Foundation) via the grant SFB-TRR 358/1 2023-491392403. B. Sun was supported in part by  National Key R \& D Program of China No. 2022YFA1005300 and New Cornerstone Science Foundation.

\section{Smooth representations of compact Lie groups}
From now on, $G$ is a (possibly disconnected) compact Lie group with complexified Lie algebra $\g$, and $(\pi,V)$ is a smooth representation of $G$. Fix a  $G$-invariant, non-degenerate, symmetric bilinear form $B$ on $\g$ whose restriction to $\Lie(G)$ is real valued and  negative definite.

Write $\wh{G}$ for the set of isomorphism classes of irreducible unitary representations of $G$. For every $\lambda\in\wh{G}$, fix an irreducible representation $(\pi_\lambda,V_\lambda)$ of class $\lambda$. Then the Casimir operator $\Omega_\g$ with respect to $B$ acts on $V_\lambda$ via the scalar multiplication by a non-negative real scalar $c(\lambda)$.

The normalized character $\chi_\lambda(g)=\dim V_\lambda \cdot \Tr \pi_\lambda (g)$ is an idempotent in the convolution algebra $C^\infty(G)$ with respect to the normalized Haar measure, which acts on the smooth representation $(\pi,V)$ of $G$ canonically. Denote by $V(\lambda)$  the $\lambda$-isotypic component of $V$, which is automatically closed in $V$. Then $\pi(\ov\chi_\lambda): V\rightarrow V$ is a continuous projection onto $V(\lambda)$. Here and henceforth, $``\,\ov{\phantom a}\,"$ over a character denotes the complex conjugation.

The following theorem is proved by Harish-Chandra (see \cite[Theorem 4.4.2.1]{Warner12}) under the assumption that $V$ is complete. The same proof works for general quasi-complete spaces as well.
\begin{thml}
	For every $v\in V$, the Fourier series
	\be
	\sum_{\lambda\in \wh{G}} \pi(\ov\chi_\lambda)v
	\ee
	converges absolutely to $v$.
\end{thml}

As a consequence of the above theorem, we have an identification
\begin{equation}\label{idv00}
	V=\left\{
	\left(v(\lambda)\right)_\lambda\in \prod_{\lambda\in\wh{G}}V(\lambda)
	\mid
	\sum_{\lambda\in \wh{G}} v(\lambda) \text{ is absolutely convergent }      \right\}.
\end{equation}

\begin{leml}\label{prop}
	Let $W\subseteq V$ be a subspace. Assume that for every $\lambda\in\wh{G}$, the subspace $W(\lambda):=\pi(\ov\chi_\lambda)W$ is closed in $V$ and contained in $W$.
	Then under the identification \eqref{idv00}, the closure of $W$ in $V$ equals
	\bee
	\left\{\left(w(\lambda)\right)_\lambda\in \prod_{\lambda\in\wh{G}}W(\lambda)\mid \sum_{\lambda\in \wh{G}} w(\lambda) \text{ is absolutely convergent }\right\}.
	\eee
	Consequently,
	\bee
 \ov{W}= \ov{\bigoplus_{\lambda\in \wh{G}}W(\lambda)}.
	\eee
	Here, $``\,\ov{\phantom a}\,"$ over subsets of $V$ denotes the closure in $V$.
\end{leml}
\begin{proof}
	We have that 
	\[
	\pi(\ov \chi_\lambda)(\ov W)\subset \ov{\pi(\ov \chi_\lambda)( W)}=\ov{W(\lambda)}=W(\lambda).
	\]
	This implies the lemma. 
\end{proof}

As usual we write  $\BN:=\{0,1,2,\dots\}$. We conclude this section by quoting two lemmas that will be used in Section \ref{sec:proof} for proving the convergence of some Fourier series.

\begin{leml}\cite[Lemma 4.4.2.2]{Warner12}\label{lemma1}
	For every continuous seminorm $\abs{\,\cdot\,}_p$ on $V$, there exists a continuous seminorm $\abs{\,\cdot\,}_q$ on $V$ such that
	\be
	\abs{\pi\left(\ov\chi_\lambda\right) v}_p\leq \left(1+c\left(\lambda\right)\right)^{-m}\cdot (\dim V_\lambda)^2\cdot \abs{\left(1+\Omega_\g\right)^m v}_q
	\ee
	for all $m\in \BN$,  $v\in V$, and $\lambda\in \wh{G}$.
\end{leml}

\begin{leml}\label{lemma2}
	For every $n\in \BN$, there exists an $m\in \BN$ such that the series
	\[
	\sum_{\lambda\in \wh{G}} (\dim V_\lambda)^n (1+c(\lambda))^{-m}
	\]
	converges absolutely.
\end{leml}
\begin{proof}
	This lemma is a variation of \cite[Lemma 4.4.2.3]{Warner12}, whose proof also works in our context.
\end{proof}

\section{Cubic Dirac operator and $\u$-cohomology}\label{sec:Dirac op}

In this section, we review the definition and basic properties of Kostant's cubic Dirac operator. For more details about Dirac operators, we refer to \cite{Huang06,HPR05,Ko99}.

We retain the notations from previous sections. In particular, $\g$ is the complexified Lie algebra of a compact Lie group $G$, $\q=\l\oplus\u$ is a parabolic subalgebra and $L=N_G(\q)$ is the Levi subgroup.

From this point onward, we assume that the bilinear form $B$ extends the Killing form on the semisimple part $\g^{ss}$ of $\g$. 

From now on, we fix a Cartan subalgebra $\h$ of $\g$ contained in $\l$. Then $B$ induces on $\h^*$ a non-degenerate bilinear form which we denote by $(\cdot,\cdot)$. Write $\Delta_\g=\Delta(\g, \h)$ and $\Delta_\l=\Delta(\l,\h)$ for the root systems of $\g$ and $\l$ with respect to $\h$, respectively. Write $\Delta(\u)$ for the set of roots of $\h$ occurring in $\u$. Fix a positive root system $\Delta_\g^+$ containing $\Delta(\u)$. Then $\Delta_\l^+: =\Delta_\l\cap \Delta_\g^+$ is a positive root system of $\l$ and $\Delta_\g^+=\Delta_\l^+\cup\Delta(\u)$. Write $\rho_\g$ and $\rho_\l$ for the half sums of the roots in $ \Delta_\g^+$ and $\Delta_\l^+$, respectively. We will use the notation $\rho(\u)$ for the half sum of the roots occurring in $\u$. Then $\rho_\g=\rho_\l+\rho(\u)$.

There is an orthogonal decomposition $\g=\l\oplus\s$ with $\s=\u\oplus\ov\u$.
Note that the restriction of $B$ to $\s$ is non-degenerate. 

Recall the Clifford algebra 
\[
\RC(\s):=\RT(\s)/(\textrm{the ideal generated by } \{x\otimes x+B(x,x):x\in\s\}),
\] 
where $\RT(\s)$ is the tensor algebra over $\s$. 
Consequently, $\RC(\s)$ is generated by $\s$ with the following defining relations:
\be\label{relation in Clifford alge}
x\cdot x= - B(x,x),\quad x\in \s. 
\ee

Since $B$ is $G$-invariant, the adjoint action of $\l$ on $\s$ induces a Lie algebra homomorphism
\[
\l \rightarrow \s\o (\s).
\]
Compositing it with the embedding of $\s\o(\s)$ into the Clifford algebra $\RC(\s)$ (cf.\, \cite[Section 2.1.9]{Huang06}), we obtain a Lie algebra homomorphism
\[
\nu:\l \rightarrow \RC(\s).
\]
Then we embed the Lie algebra $\l$ into $\RU(\g)\otimes \RC(\s)$ via
\[
X\longmapsto X\otimes 1+1\otimes \nu(X),
\]
where $\RU$ indicates the universal enveloping algebra. This embedding extends to an algebra homomorphism
\be\label{embeddingofl}
\gamma\colon \RU(\l)\longrightarrow \RU(\g)\otimes \RC(\s).
\ee

 We embed the exterior algebra $\wedge^\bullet \s$ into $\RT(\s)$ as a subspace via
\[
x_1\wedge\cdots \wedge x_k \longmapsto \frac{1}{k!}\sum_{\sigma\in S_k}\sgn(\sigma) x_{\sigma(1)}\otimes\cdots\otimes x_{\sigma(k)} \quad (k\in \mathbb{N}),
\]
where $S_k$ is the group of permutations on $\{1,\cdots,k\}$, $\sgn(\sigma)$ is the sign of a permutation $\sigma$.

The Chevalley map $\varphi: \wedge^\bullet \s\rightarrow \RC(\s)$ is obtained by composing this embedding with the natural homomorphism $\RT(\s)\rightarrow \RC(\s)$. Let $\{Z_1,\dots,Z_n\}$ be an orthonormal basis of $\s$. Then the Chevalley map $\varphi$ is the linear map determined by the following formulas:
\[
\varphi (Z_{i_1}\wedge\cdots\wedge Z_{i_k})= Z_{i_1}\cdots Z_{i_k} 
\] 
where $1\leq i_1<\cdots<i_k\leq n.$ The Chevalley map is an isomorphism of vector spaces.

 Note that the bilinear form $B$ induces a degree-preserving identification $\wedge^\bullet \s\cong\wedge^\bullet (\s^*)$. Let $v\in\wedge^3\s$ be the element corresponding to the $3$-form $\omega \in \wedge^3 (\s^*) $ such that 
 \be\label{dfn:cubicterm}
 2\omega(X,Y,Z)=B([X,Y],Z)
 \ee
 for all $X,Y,Z\in \s$. To be explicit,
 \be\label{eq:cubicterm}
 v=\cfrac{1}{2}\sum_{1\leq i<j<k\leq n} B\left([Z_i,Z_j],Z_k\right)Z_i\wedge Z_j\wedge Z_k.
 \ee
The Kostant's cubic Dirac operator $D$ with respect to $(\g,\l)$ is the element 
\[
D=\sum_{i=1}^{n} Z_i\otimes Z_i+1\otimes \varphi(v)\in \mathrm \RU(\g)\otimes \mathrm \RC(\s).
\]
The definition of $D$ is independent with the choice of the orthonormal basis. 

The following result, which is crucial for our application, is proved by Kostant in {\cite[Theorem 2.16]{Ko99}} in a more general setting. 

\begin{thml}\label{square}
	The equality 
	\be\label{esq}
	D^2=-\Omega_\g\otimes 1+\gamma(\Omega_\l)+ \con
	\ee
	holds in $\RU(\g)\otimes \RC(\s)$, where $\Omega_\g$ and $\Omega_\l$ are the Casimir elements for $\g$ and $\l$ with respect to $B$,  and
	\[
	\con=(\rho_\l,\rho_\l)-(\rho_\g,\rho_\g),
	\]
	where $\rho_{\g},\rho_{\l}\in \h^*$ are the half sums of positive roots of $\g,\l$, respectively.
\end{thml}

\begin{rmkl}
	Our definition of Clifford algebra follows the one in \cite{Huang06}. However, Kostant uses a different definition of the Clifford algebra $\RC(\s)$ in \cite{Ko99}, which requires $x \cdot x = B(x, x)$ for $x \in \s$.  Therefore, equalities \eqref{dfn:cubicterm} and \eqref{esq} are respectively different from  \cite[Formula (1.20)]{Ko99} and \cite[Formula (2.17)]{Ko99} by some signs. 
\end{rmkl}

Since $\s=\u\oplus\ov\u$ is even-dimensional, the Clifford algebra $\RC(\s)$ has a unique irreducible module, which is called the Spin module of $\RC(\s)$.  
Fix a nonzero element $\ov u_{\mathrm{top}}$ in $\wedge^{\mathrm{top}}\ov\u$, and define a linear isomorphism
\be\label{eq:Spin}
\wedge^\bullet\u\longrightarrow \varphi(\wedge^\bullet \u\otimes \wedge^{\mathrm{top}}\ov\u) \qquad x\mapsto \varphi(x\otimes \ov u_{\mathrm{top}})
\ee 
where $\varphi$ is the Chevalley map, and $\wedge^\bullet \u\otimes \wedge^{\mathrm{top}}\ov\u$ is a subspace of $\wedge^\bullet\s=\wedge^\bullet \u\otimes \wedge^\bullet\ov\u$. Note that $\varphi(\wedge^\bullet \u\otimes \wedge^{\mathrm{top}}\ov\u)$ is the left ideal of $\RC(\s)$ generated by $\varphi(\wedge^{\mathrm{top}}\ov\u) $. 
Denote by $S$ the vector space $\wedge^\bullet\u$, together with the $\RC(\s)$-action induced by the linear isomorphism \eqref{eq:Spin}. 
Then $S$ is a Spin module of $\RC(\s)$ (see \cite[Section 2.2.2]{Huang06}).

The Killing form (and hence also $B$) induces a non-degenerate pairing $\u\times\ov\u\to\C$ allowing us to identify the dual of $\ov\u$ with $\u$ in a canonical way. In particular, for a smooth representation $V$ of $G$, one obtains a canonical degree-preserving, $L$-equivariant isomorphism 
\begin{equation}
	\wedge^\bullet\u\otimes V\longrightarrow \wedge^\bullet (\ov\u^*)\otimes V
	\label{eq:dualcomplexisomorphism}
\end{equation}
between complexes computing $\RH_\bullet(\u,V)$ and $\RH^\bullet(\ov\u,V)$, which allows us to consider the coboundary map $\dif\,$ on the right hand side and the boundary map $\partial$ on the left hand side as being defined on the same space $\wedge^\bullet\u\otimes V=V\otimes S$.

\begin{thml}\cite[Proposition 9.1.6]{Huang06}\label{thm:huang}
	Under the action of $\RU(\g)\otimes \RC(\s)$ on $V\otimes S$, the cubic Dirac operator D acts on $V\otimes S $ as $2\partial+\dif\,$.
\end{thml}

\begin{thml}\label{decom}
	Suppose that $V$ is finite-dimensional, the following identities hold:
	\begin{eqnarray}
		&V\otimes S= \Ker\, D\oplus \Im\, D,& \label{eq:decom1} \\
		&\Ker\, D =\Ker\, D^2,\qquad\,\,\,
		&\Im\, D  =\Im\, D^2,  \label{eq:decom2}\\
		&\Ker\, D = \Ker \dif\, \cap\, \Ker\,\partial ,
		&\Im\, D  = \Im \dif\,\oplus\, \Im\,\partial,\label{eq:decom3}\\
		&\Ker \dif\,=\Ker\, D\oplus \Im\dif,
		&\Ker\,\partial=\Ker \,D \oplus \Im\,\partial.\label{eq:decom4}
	\end{eqnarray}
	All the subspaces of $V\otimes S$ above are $D^2$-invariant. Moreover, the restriction of $D^2$ to $\Im \dif$ is bijective and equals $2\!\dif\!\partial$. The restriction of $D^2$ to $\Im\,\partial$ is bijective and equals $2 \partial\!\dif$.
\end{thml}

\begin{proof}
The identities in \eqref{eq:decom1} and \eqref{eq:decom2} are established in \cite[Lemma 9.2.4]{Huang06}. The identities in \eqref{eq:decom3} follow from \cite[Lemma 9.2.3]{Huang06} and \cite[Formula (9.3)]{Huang06}. The identities in \eqref{eq:decom4} are proved in \cite[Theorem 9.2.5]{Huang06}. Note that the second identity in \eqref{eq:decom2} implies that the restriction of $D^2$ to $\Im \,D$ is bijective. Then the rest of the theorem is easily deduced from the formula $D^2=2\partial\!\dif\!+ 2\!\dif\!\partial.$
\end{proof}

\section{Proof of the main theorem}\label{sec:proof}

We keep the notations and assumptions from previous sections. In particular, $G$ is a (possibly disconnected) compact Lie group  with complexified Lie algebra $\g$, $(\pi,V)$ is a smooth representation of $G$ and $\q$ is a parabolic subalgebra of $\g$ with nilpotent radical $\u$.
Use the notations in Section \ref{sec:Dirac op}, $S=\wedge^\bullet\u$, $\h$ is a common Cartan subalgebra of $\l$ and $\g$, $D$ is the cubic Dirac operator with respect to $(\g,\l)$ and $\partial \in \End(V\otimes S)$ (resp. $\dif\in \End(V\otimes S)$) are the boundary (resp. coboundary) map of Lie algebra homology (resp. cohomology). From now on, we view $D$ as an operator on $V\otimes S$.

For every $\lambda\in\wh{G}$, we already fixed a representative $V_\lambda$, which has finitely many highest weights as a $\g$-module with respect to the positive root system $\Delta_\g^+$. We choose one of these highest weights and denote it by $\lambda$, by abuse of notation. 
Similarly, we fix for every $\mu \in \wh{L}$ in the unitary dual of $L$ a representative $W_\mu$ and a highest weight of $W_\mu$, which is denoted by $\mu$.

We view $S$ as a representation of $G$ with trivial action, then $V\otimes S$ is also a smooth representation of $G$, and this action of $G$ on $V\otimes S$ is denoted by $\Pi$.

We consider a representation  $(\Sigma,V\otimes S)$ of $L$, where the action $\Sigma$ is given by the tensor product of the restriction of $\pi$ and the adjoint action on $S=\wedge^\bullet \u.$ It is worth mentioning that $\dif\,,\partial,D$ are $L$-equivariant operators on $(\Sigma,V\otimes S)$. 

For $\lambda\in \wh{G}$ and $\mu\in \wh{L}$, consider the corresponding continuous linear projectors onto the $\lambda$- and $\mu$-isotypic components of $V\otimes S$:
\begin{eqnarray*}
	&P(\lambda):= \Pi(\ov\chi_\lambda):&V\otimes S\longrightarrow V\otimes S,\\
	&Q(\mu):= \Sigma(\ov\chi_\mu):& V\otimes S \longrightarrow V\otimes S.
\end{eqnarray*}

\begin{lem}\label{lem:elementary}
	The element $\Omega_\g\otimes 1\in \RU(\g)\otimes \RC(\s)$ acts on $\Im\, P(\lambda)$ via the scalar multiplication by
	\[
	c(\lambda)=(\lambda+\rho_\g,\lambda+\rho_\g)-(\rho_\g,\rho_\g),
	\]
	and $\gamma(\Omega_\l)$ acts on $\Im\, Q(\mu)$ via the scalar multiplication by
	\[
	(\mu+\rho(\ov\u)+\rho_\l,\mu+\rho(\ov\u)+\rho_\l)-(\rho_\l,\rho_\l).
	\]
\end{lem}
\begin{proof}
   It is easy to prove that $\Omega_\g$ acts on $V_\lambda$ via the scalar multiplication by
  \[
  c(\lambda)=(\lambda+\rho_\g,\lambda+\rho_\g)-(\rho_\g,\rho_\g).
  \]
  It follows that the element $\Omega_\g\otimes 1\in \RU(\g)\otimes \RC(\s)$ acts on $\Im P(\lambda)$ via the scalar multiplication by $c(\lambda)$.
  
  By differentiating the group action $\Sigma$, we obtain the corresponding action of $\l$ on $V\otimes S$. By abuse of notation, we also denote this action by $\Sigma$.
  The $\RU(\g)\otimes\RC(\s)$-module structure on $V\otimes S$, combined with the morphism $\gamma:\RU(\l)\rightarrow\RU(\g)\otimes\RC(\s)$ (cf. \eqref{embeddingofl}), induces another action of $\l$ on $V\otimes S$. By abuse of notation, we denote this action by $\gamma$. It is proved in \cite[Proposisition 3.6]{Ko20} that
  \[
  \gamma\cong \Sigma\otimes \rho(\ov\u),
  \]   
as $\l$-modules, where $(\rho(\ov\u),\C_{\rho(\ov\u )})$ is the one-dimensional $\l$-module with weight $\rho(\ov\u)=\rho_\l-\rho_\g$. 

It suffices to prove that for all $\mu\in \wh{L}$, the Casimir operator $\Omega_\l$ acts on $W_\mu\otimes \C_{\rho(\ov\u )}$ via the scalar multiplication by 
\[
	(\mu+\rho(\ov\u)+\rho_\l,\mu+\rho(\ov\u)+\rho_\l)-(\rho_\l,\rho_\l).
\]
If $L$ is connected, this result is straightforward. In the general case, a bit more argument is required. 

Let $\b_L$ be the Borel subalgebra of $\l$ associated with $\Delta_\l^+$ and $T_L:=N_L(\b_L)$ be the Cartan subgroup of $L$. Recall that $T_L$ acts on $\h^*$ by the coadjoint action, which preserves $\Delta_\l^+$, and inner form $(\cdot,\cdot)$ is invariant under the coadjoint action. By the Cartan-Weyl theory (see \cite[Chapter IV, Section 2]{KV}), there is a finite set $\{t_1,\cdots,t_k \}$ of elements in the Cartan subgroup $T_L$ of $L$, such that as $\l$-modules,  
\[
W_\mu=\bigoplus_{1\leq i\leq k} m_iW'_{t_i.\mu}
\]
where $W'_{t_i.\mu}$ is the irreducible $\l$-module with highest weight $t_i.\mu$, and $m_i$ denotes the multiplicity of $W'_{t_i.\mu}$ in $W_\mu$. Then $\Omega_\l$ acts on $W'_{t_i.\mu}\otimes\C_{\rho(\ov\u )}$ via the scalar multiplication by
\[
	(t_i.\mu+\rho(\ov\u)+\rho_\l,t_i.\mu+\rho(\ov\u)+\rho_\l)-(\rho_\l,\rho_\l).
\]
Since $T_L$ normalizes $\ov\u$ and $\b_L$, it follows that $\rho(\ov\u),\rho_\l\in\h^*$ are invariant under the coadjoint action of $T_L$ on $\h^*$. Therefore,
\begin{eqnarray*}
	&&(t_i.\mu+\rho(\ov\u)+\rho_\l,t_i.\mu+\rho(\ov\u)+\rho_\l)-(\rho_\l,\rho_\l)\\
	&=&(\mu+t_i^{-1}.\rho(\ov\u)+t_i^{-1}.\rho_\l,\mu+t_i^{-1}.\rho(\ov\u)+t_i^{-1}.\rho_\l)-(\rho_\l,\rho_\l)\\
	&=&(\mu+\rho(\ov\u)+\rho_\l,\mu+\rho(\ov\u)+\rho_\l)-(\rho_\l,\rho_\l).
\end{eqnarray*}
This completes the proof of the lemma. 
\end{proof}
Note that the isotypic component $\Im\, P(\lambda)=V(\lambda)\otimes S$ is an invariant subspace for the operators $Q(\mu),\,D,\,\partial$ and $\dif$. We denote the restriction of $D,\,\partial$ and $\dif$ to endomorphisms of $\Im\, P(\lambda)$ by $D_\lambda,\,\partial_\lambda,$ and $\dif_\lambda$, respectively. Then $\partial_\lambda, \dif_\lambda \in \End(V(\lambda)\otimes S)$ are the boundary and coboundary maps for computing $\RH_\bullet(\u,V(\lambda))$ and $\RH^\bullet(\ov\u,V(\lambda))$. 

\begin{thml}\label{thm:subspacedecompositions}
 	\noindent
	(a) For all $\lambda\in\wh{G}$ and $\mu\in\wh{L}$,
	\[
	W(\lambda,\mu):=Q(\mu)\left(V(\lambda)\otimes S\right)
	\]
	is a $D^2$-invariant closed subspace of $V\otimes S$ on which $D^2$ acts via the scalar multiplication by
	\begin{equation}
		c(\lambda,\mu):=(\mu+\rho(\ov\u)+\rho_{\l},\mu+\rho(\ov\u)+\rho_{\l})-(\lambda+\rho_{\g},\lambda+\rho_{\g}).
		\label{eq:defclambdamu}
	\end{equation}
	
\noindent
(b)	As representations of $L$:
	\begin{eqnarray}
		&V(\lambda)\otimes S&=\Im\, D_\lambda\oplus\Ker\, D_\lambda  \label{de1},\\
		&\Im\, D_\lambda&=\Im\,\partial_\lambda\oplus\Im \dif_\lambda\label{de2},\\
		&\Ker \dif_\lambda&=\Im\dif_\lambda\oplus\,\Ker\, D_\lambda\label{de3},\\
		&\Ker\,\partial_\lambda&=\Im\,\partial_\lambda\oplus\Ker\, D_\lambda\label{de4},\\
		&\Im\, D_\lambda&=\Im\, D_\lambda^2\,=\bigoplus_{\mu\in I(\lambda)}W(\lambda,\mu),\label{de5}\\
		&\Ker\, D_\lambda&=\Ker\, D_\lambda^2=\bigoplus_{\mu\in J(\lambda)}W(\lambda,\mu)\label{de6},
	\end{eqnarray}
where
	\begin{eqnarray*}
		&I(\lambda):=\{\mu\in\wh{L}: \Hom_L(W_\mu,V_\lambda\otimes S)\not= 0 \text{ and } c(\lambda,\mu)\not=0\},\\
		&J(\lambda):=\{\mu\in\wh{L}: \Hom_L(W_\mu,V_\lambda\otimes S)\not= 0 \text{ and } c(\lambda,\mu)=0\}.
	\end{eqnarray*}
	
	\noindent
(c)	All subspaces of $V\otimes S$ occurring in \eqref{de1}---\eqref{de6} are closed and $D^2$-invariant. 
Moreover, the restriction of $D_\lambda^2$ to $\Im\,\partial_\lambda$ is bijective and equals $2\partial_\lambda\!\dif_\lambda$, and the restriction of $D_\lambda^2$ to $\Im\,\dif_\lambda$ is bijective and equals $2\!\dif_\lambda\!\partial_\lambda$. 
\end{thml}

\begin{proof}
	Combining Theorem \ref{square} with Lemma \ref{lem:elementary}, we have that $D^2$ acts on $W(\lambda,\mu)$ via the scalar multiplication by $c(\lambda,\mu)=(\mu+\rho(\ov\u)+\rho_{\l},\mu+\rho(\ov\u)+\rho_{\l})-(\lambda+\rho_{\g},\lambda+\rho_{\g})$, which in particular shows that $W(\lambda,\mu)$ is stable under $D^2$.
	
	We impose $\Hom_G(V_\lambda,V)$ with the trivial actions of $G$ and $\RU(\g)\otimes \RC(\s)$. It is clear that $V(\lambda)=V_\lambda\otimes \Hom_G(V_\lambda,V)$ as smooth representations of $G$, which implies that
	\[
	V(\lambda)\otimes S=(V_\lambda\otimes S)\otimes\Hom_G(V_\lambda,V)
	\]
as smooth representations of $G$ and as $\RU(\g)\otimes \RC(\s)$-modules.
	  For every subspace $V'\subseteq V_\lambda\otimes S$ the corresponding subspace $V'\otimes\Hom_G(V_\lambda,V)$ is closed. Therefore, applying Theorem \ref{decom} to the factor $V_\lambda\otimes S$, applying $-\otimes\Hom_G(V_\lambda,V)$ and translating the result back to $V(\lambda)\otimes S$, one obtains the identities \eqref{de1}, \eqref{de2}, \eqref{de3} and \eqref{de4}, as well as the respective first identities in \eqref{de5} and \eqref{de6}. Moreover, all spaces which occur are closed in $V\otimes S$.

	In order to establish the respective second identities in \eqref{de5} and \eqref{de6}, observe that the sets $I(\lambda)$ and $J(\lambda)$ are finite. Consequently,
	\[
	V(\lambda)\otimes S=\bigoplus_{\mu\in I(\lambda)\cup J(\lambda)}W(\lambda,\mu)
	\]
	is a finite direct sum. The proof of \eqref{de5} and \eqref{de6} is completed by combining this decomposition with the known action of $D_\lambda^2$ via the scalar $c(\lambda,\mu)$.
\end{proof}

\begin{lem}\label{lem:bounds}
	For all $\lambda\in\wh{G}$,
	\begin{equation}
		\sum_{\mu\in I(\lambda)}(\dim W_\mu)^2\leq (\dim V_\lambda)^2(\dim S)^2.
		\label{eq:dimensionestimate}
	\end{equation}
	Moreover, there exists a constant $c_\g>0$, only depending on $\g$, with the property that for all $\lambda\in\wh{G}$ and all $\mu\in I(\lambda)$:
	\begin{equation}
		c_\g\leq \abs{c(\lambda,\mu)}.
		\label{eq:cmulambdabound}
	\end{equation}
\end{lem}

\begin{proof}
	In the isotypic decomposition of $ V_\lambda\otimes S$ as a representation of $L$, the class of $W_\mu$ occurs at least once if and only if $\mu\in I(\mu)\cup J(\mu)$. Therefore
	\[
	\sum_{\mu\in I(\lambda)}\dim W_\mu\leq \dim V_\lambda\otimes S.
	\]
	Squaring both sides of the inequality and applying the Cauchy-Schwartz inequality yields \eqref{eq:dimensionestimate}.
	
	There is an orthogonal decomposition $\h^*=(\h^{ss})^*\oplus \z(\g)^*$, where $\h^{ss}:=\h\cap\g^{ss}$ and $\z(\g)$ is the center of $\g$. Let $\Lambda_{ss}\subseteq (\h^{ss})^*\subseteq \h^*$ denote the lattice of integral weights of $\g^{ss}$. 
	Define a real-valued function $f$ on $\Lambda_{ss}\times\Lambda_{ss}$ by
	\[
	f(x,y)=(y,y)-(x,x),\quad (x,y)\in\Lambda_{ss}\times\Lambda_{ss}.
	\]
	Note that $B$ extends the Killing form on $\g^{ss}$, which implies that $(\cdot,\cdot)$ is rational with respect to the weight lattice $\Lambda_{ss}$.  
	It follows that $f(x,y)$ is a polynomial with rational coefficients on $\Lambda_{ss}\times\Lambda_{ss}$. Consequently, there exists a constant $c_\g>0$ such that 
	$|f(x,y)|\geq c_\g$ whenever $f(x,y)\neq 0$.	
	
	For $x \in \h^*$, write $x_{ss}$ and $x_{\z(\g)}$ for its restrictions 
	to $\h^{ss}$ and $\z(\g)$, respectively.
	It is clear that for all $\lambda\in \wh G$ and all $\mu\in I(\lambda)$, 
	\[
	\lambda_{\z(\g)}=\mu_{\z(\g)} \quad\text{and}\quad c(\lambda,\mu)=c(\lambda_{ss},\mu_{ss})=f(\lambda_{ss}+\rho_{\g}, \mu_{ss}+\rho(\ov\u)+\rho_{\l})\neq 0.
	\]
	Consequently, $\abs{c(\lambda,\mu)}\geq c_\g$ for all $\lambda\in \wh G$ and all $\mu\in I(\lambda)$.
\end{proof}
	
For all $\lambda\in\wh G$, by decomposition \eqref{de5}, one knows that $D_\lambda^2$ induces an (algebraic) automorphism of $\Im\, D_\lambda$. Using decomposition \eqref{de1}, let $T_\lambda$ be the linear operator on $V(\lambda)\otimes S$, which is inverse of $D_\lambda^2$ on $\Im \,D_\lambda$ and acts as the zero operator on $\Ker D_\lambda$. To be explicit, the formula of $T_\lambda$ is given by
\be \label{eq:defofTlambda}
T_\lambda v(\lambda)=\sum_{\mu\in I(\lambda)} \frac{1}{c(\lambda,\mu)}Q(\mu)v(\lambda)
\ee
for all $v(\lambda)\in V(\lambda)\otimes S$. It is clear that $T_\lambda$ is $L$-equivariant and
\be\label{eq:D2Tlambda}
D^2_\lambda T_\lambda =T_\lambda  D^2_\lambda =(T_\lambda D^2_\lambda)^2.
\ee

\begin{lem}\label{est}
	For every continuous seminorm $\abs{\,\cdot\,}_p$ on $V\otimes S$,  there is a continuous seminorm $\abs{\,\cdot\,}_q$ such that
	\begin{eqnarray}
		&&\abs{T_\lambda v(\lambda)}_p\leq c_\g^{-1}(\dim S)^2(\dim V_\lambda)^2 \abs{v(\lambda)}_q
	\end{eqnarray}
	for all $\lambda\in\wh G$, $v(\lambda)\in V(\lambda)\otimes S$, where $c_\g>0$ is the constant in Lemma \ref{lem:bounds}. In particular, $T_\lambda$ is continuous.
\end{lem}
\begin{proof}
	Applying the bound \eqref{eq:cmulambdabound} from Lemma \ref{lem:bounds} to \eqref{eq:defofTlambda}, we obtain
	\begin{equation}\label{ineq1}
		\abs{T_\lambda v(\lambda)}_p\leq c_\g^{-1}\sum_{\mu \in I(\lambda)} \abs{Q(\mu)v(\lambda)}_p
	\end{equation}
	for all $v(\lambda)\in V(\lambda)\otimes S$. 
	
	Invoking Lemma \ref{lemma1} for the representation $(\Sigma, V\otimes S)$, there is a continuous seminorm $\abs{\,\cdot\,}_q$ on $V\otimes S$ such that
	\[
	\abs{Q(\mu)v}_p\leq (\dim W_\mu)^2\abs{v}_q
	\]
	for all $v\in V\otimes S$.
	
	Summing over all $\mu\in I(\lambda)$ and applying the bound \eqref{eq:dimensionestimate} from Lemma \ref{lem:bounds}, we obtain
	\be\label{ineq2}
	\sum_{\mu \in I(\lambda)} \abs{Q(\mu)v(\lambda)}_p \leq (\dim V_\lambda)^2(\dim S)^2\abs{v(\lambda)}_q
	\ee
	for all $v(\lambda)\in V(\lambda)\otimes S$.
	
	Combining \eqref{ineq1} and \eqref{ineq2}, we obtain
	\begin{equation}
		\abs{T_\lambda v(\lambda)}_p\leq c_\g^{-1}(\dim V_\lambda)^2(\dim S)^2\abs{v(\lambda)}_q.
	\end{equation}
	Therefore, $T_\lambda$ is continuous.
\end{proof}

We augment the family of operators $T_\lambda$ for $\lambda\in\wh{G}$ to an operator on $V\otimes S$.

\begin{lem}\label{lem:existenceofT}
	There is a unique $L$-equivariant continuous linear operator $T$ on $V\otimes S$ whose restriction to $V(\lambda)\otimes S$ is $T_\lambda$ for every $\lambda\in \wh{G}$. Moreover, $T$  commutes with $D^2$, and $D^2T$ is an $L$-equivariant continuous linear projection.
\end{lem}

\begin{proof}
	We aim at defining the $L$-equivariant operator $T$ by the formula
	\be\label{eq:defofT}
	Tv=\sum_{\lambda\in \wh{G}}T_\lambda P(\lambda)v,\quad \text{for~} v\in V\otimes S.
	\ee
	For if $T$ exists, it must be of this form. We need to show that $T$ is well defined. That is the right hand side is absolutely convergent for every $v\in V\otimes S$ and that the resulting operator $T$ is continuous.
	
	Let $\abs{\,\cdot\,}_p$ be a continuous seminorm on $V\otimes S$, by Lemma \ref{est}, there is a continuous seminorm $\abs{\,\cdot\,}_q$ such that for all $v\in V\otimes S$,
	\begin{eqnarray*}
		\abs{T_\lambda P(\lambda)v}_p\leq c_\g^{-1} (\dim S)^2(\dim V_\lambda)^2 \abs{P(\lambda)v}_q
	\end{eqnarray*}
	where $c_\g$ is a positive constant depending only on $\g$.
	
	Using Lemma \ref{lemma1}, there is a continuous seminorm $\abs{\,\cdot\,}_r$ such that for all $m\in\BN$ and all $v\in V\otimes S$,
	\[
	\abs{P(\lambda)v}_q\leq (1+c(\lambda))^{-m}(\dim V_\lambda)^2\abs{(1+\Omega_\g)^mv}_r.
	\]
	Hence, for all $m\in \BN$ and all $ v\in V\otimes S$
	\begin{eqnarray*}
		\sum_{\lambda\in \wh{G}}\abs{T_\lambda P(\lambda)v}_p&\leq&
		c_\g^{-1} \sum_{\lambda\in \wh{G}}(\dim S)^2(\dim V_\lambda)^2 \abs{P(\lambda)v}_q\\
		&\leq& \left(\sum_{\lambda\in \wh{G}}  (\dim V_\lambda)^4(1+c(\lambda))^{-m}\right) c_\g^{-1}\cdot (\dim S)^2\cdot \abs{(1+\Omega_\g)^mv}_r.
	\end{eqnarray*}
	By Lemma \ref{lemma2}, there is an $m_0\in \BN$ such that 
	\[
	C_0:=c_\g^{-1}\cdot (\dim S)^2\sum_{\lambda\in \wh{G}}  (\dim V_\lambda)^4(1+c(\lambda))^{-m}<\infty.
	\]
	We conclude that 
	\begin{equation}\label{convergence}
		\sum_{\lambda\in \wh{G}}\abs{T_\lambda P(\lambda)v}_p\leq C_0 \abs{(1+\Omega_\g)^{m_0}v}_r.
	\end{equation}
	Since $\abs{\,\cdot\,}_p$ is an arbitrary continuous seminorm on $V\otimes S$, the Fourier series \[
	\sum_{\lambda\in \wh{G}}T_\lambda 	P(\lambda)v
	\]
	is absolutely convergent and \eqref{eq:defofT} defines a linear operator $T$ on $V\otimes S$. Moreover, for every continuous seminorm $\abs{\,\cdot\,}_p$, we have 
	\begin{equation}
		\abs{Tv}_p\leq\sum_{\lambda\in \wh{G}}\abs{T_\lambda P(\lambda)v}_p\leq C_0 \abs{(1+\Omega_\g)^{m_0}v}_r=C_0 \abs{v}_t,
	\end{equation}
	where $\abs{\,\cdot\,}_t$ is the continuous seminorm given by $\abs{v}_t=\abs{(1+\Omega_\g)^{m_0}v}_r$. This establishes the continuity of $T$.
	
	By construction of $T$ (cf.\ \eqref{eq:defofT}), the relations $D_\lambda^2T_\lambda=T_\lambda D_\lambda^2=(T_\lambda D_\lambda^2)^2$ (cf.\,\eqref{eq:D2Tlambda}) imply the relations $D^2T=TD^2=(T D^2)^2$.
\end{proof}

With the $L$-equivariant continuous projection $D^2T$ at hand, we now proceed to prove our main theorem. 
\begin{thml}
	Let $G$ be a compact Lie group and $V$ be a smooth representation of $G$. Then subspaces $\Im\, D,\Im\,\partial,\Im\dif\!$ of $V\otimes S$ are closed. Moreover,  the following decompositions of smooth representations of $L$ hold:
	\begin{eqnarray}
		V\otimes S&=&\Ker \,D\oplus\Im\, D,\label{de7}\\
		\Ker\,\partial&=&\Ker\, D\oplus\Im\,\partial,\label{de8}\\
		\Ker \dif&=&\Ker\, D\oplus\Im\dif\!.\label{de9}
	\end{eqnarray}
	Therefore, $\RH_\bullet(\u,V)$ and $\RH^\bullet(\ov\u,V)$ are Hausdorff and quasi-complete. Moreover, 
	\be\label{eq:dualhomologyiso}
	\RH_\bullet(\u,V)\cong \RH^\bullet(\ov\u,V)\cong \Ker\, D
	\ee
	 as smooth representations of $L$.	
\end{thml}

\begin{proof}
	By Lemma \ref{lem:existenceofT}, $D^2T$ is an $L$-equivariant continuous projection, therefore the image and kernel of $D^2T$ are closed and we have a decomposition of smooth representations of $L$:
	\bee
	V\otimes S= \Ker\, D^2T\oplus \Im\, D^2T.
	\eee
	
	On the one hand, Lemma \ref{prop} together with $\Im\, D_\lambda^2T_\lambda=\Im\,D_\lambda^2=\Im\, D_\lambda$ (cf.\ \eqref{de5} and \eqref{eq:D2Tlambda}) shows that
	\[
	\Im\, D^2T=\ov{\bigoplus_{\lambda\in \wh{G}}\Im\, D_\lambda^2T_\lambda}=\ov{\bigoplus_{\lambda\in \wh{G}}\Im\, D_\lambda}=\ov{\Im\, D}.
	\]
	In this proof, $``\,\ov{\phantom a}\,"$  over a subspace of $V\otimes S$ denotes its closure in $V\otimes S$.
	On the other hand, we have the chain of inclusions
	\[
	\Im\, D^2T\subset \Im\, D^2\subset \Im\, D\subset \ov{\Im\, D}=\Im\, D^2T.
	\]
	We conclude that $\Im\, D=\Im\, D^2T$ is closed in $V\otimes S$.
	
	Combining Theorem \ref{thm:subspacedecompositions} with \eqref{eq:D2Tlambda}, 
	 it follows that
	\[
	\Ker\, D_\lambda=\Ker\, D_\lambda^2 =\Ker\, D^2_\lambda T_\lambda=\Ker \dif_\lambda\cap\,\Ker\,\partial_\lambda.
	\]
	Applying Lemma \ref{prop} once more, we conclude that
	\be\label{eq:kerD}
	\Ker\, D=\Ker \,D^2T=\Ker \dif\cap\, \Ker\,\partial.
	\ee
	Therefore, we obtain the decomposition \eqref{de7} of  smooth representations of $L$ as claimed.
	
	Note that $\partial_\lambda,\dif_\lambda$ commute with $T_\lambda,D^2_\lambda$ for all $\lambda\in\wh G$. It follows that $\partial,\dif$ commute with $T,D^2$. Then $\Ker\,\partial$ and $\Ker \dif$ are closed subspaces which are invariant under $D^2$ and $T$. In addition, $\Ker\,\partial$ and $\Ker \dif$ are $L$-invariant subspaces. Therefore $D^2T$ restricted to $\Ker\,\partial$ and $\Ker\dif\,$ is again an $L$-equivariant continuous projection operator on $\Ker\,\partial$ and $\Ker \dif$. Combining this fact with \eqref{eq:kerD} shows that
	\begin{eqnarray*}
		\Ker\,\partial&=& \Ker\, (D^2T|_{\Ker\,\partial})\oplus \Im\, (D^2T|_{\Ker\,\partial})= \Ker\, D\oplus D^2T(\Ker\,\partial),\\
		\Ker \dif&=& \Ker\, (D^2T|_{\Ker \dif})\oplus \Im\, (D^2T|_{\Ker \dif})=\Ker\, D\oplus D^2T(\Ker \dif),
	\end{eqnarray*}
	as smooth representations of $L$.
	
	By Theorem \ref{thm:subspacedecompositions},  identities \eqref{de3} \eqref{de4} therein and Lemma \ref{prop}, we have 
	\begin{eqnarray*}
		\Im\, (D^2T|_{\Ker\,\partial})
		&=&D^2T(\Ker\,\partial)
		=	\ov{\bigoplus_{\lambda\in\wh{G}} D^2_\lambda T_\lambda(\Ker\,\partial_\lambda) }
		=\ov{\bigoplus_{\lambda\in\wh{G}}\Im\,\partial_\lambda\,}
		=\ov{\Im\,\partial},\\
		\Im\, (D^2T|_{\Ker\dif})
		&=&D^2T(\Ker\dif)
		=	\ov{\bigoplus_{\lambda\in\wh{G}} D^2_\lambda T_\lambda(\Ker\dif_\lambda) }
		=\ov{\bigoplus_{\lambda\in\wh{G}}\Im\dif_\lambda}
		=\ov{\Im \dif}.
	\end{eqnarray*}
	By Theorem \ref{thm:huang}, we have $D^2=2\partial\!\dif+2\!\dif\!\partial$, which in turn implies
	\[
	D^2|_{\Ker\,\partial}=2\partial\!\dif\!|_{\Ker\,\partial}\quad\text{and}\quad D^2|_{\Ker \dif}=2\!\dif\!\partial|_{\Ker \dif}.
	\]
	Consequently, there are chains of inclusions
\begin{eqnarray*}
	&&\Im\,\partial\subset \overline{\Im\,\partial} \,=D^2T(\Ker\,\partial)=2\partial\!\dif T(\Ker\,\partial)\subset \Im\,\partial,\\
	&&\Im\dif\,\subset \overline{\Im\dif}=D^2T(\Ker \dif)\,=2\!\dif\!\partial T(\Ker\dif)\subset \Im \dif\!.	
\end{eqnarray*}
	This proves decompositions \eqref{de8} and \eqref{de9} of smooth representations of $L$.
\end{proof}

\end{document}